\documentclass[runningheads]{llncs}
\usepackage{graphicx}
\usepackage{amsmath}
\usepackage{amssymb}
\usepackage{textgreek}
\usepackage{tikz}
\usetikzlibrary{positioning}

\newcommand{\tuple}{\vec}

\newcommand {\A}{\mathfrak A}
\newcommand {\B}{\mathfrak B}
\newcommand {\M}{\mathfrak M}

\newcommand {\All}{\textbf{All}}

\newcommand{\FO}{\mathbf{FO}}
\newcommand{\D}{\mathbf D}

\newcommand{\NE}{\texttt{NE}}

\newcommand{\DD}{\mathcal D}

\newcommand{\Fld}{\texttt{Fld}}

\newcommand{\anon}{\mbox{\textUpsilon}}
\begin{document}
\title{Strongly First Order, Domain Independent Dependencies: the Union-Closed Case}
\titlerunning{The Union Closed Case}
%
\author{Pietro Galliani\orcidID{0000-0003-2544-5332}}
\authorrunning{P. Galliani}
%
\institute{Free University of Bozen-Bolzano
\email{pietro.galliani@unibz.it}}
\maketitle              
\begin{abstract}
Team Semantics generalizes Tarski's Semantics by defining satisfaction with respect to sets of assignments rather than with respect to single assignments. Because of this, it is possible to use Team Semantics to extend First Order Logic via new kinds of connectives or atoms -- most importantly, via \emph{dependency atoms} that express dependencies between different assignments. 

Some of these extensions are more expressive than First Order Logic proper, while others are reducible to it. In this work, I provide necessary and sufficient conditions for a dependency atom that is \emph{domain independent} (in the sense that its truth or falsity in a relation does not depend on the existence in the model of elements that do not occur in the relation) and \emph{union closed} (in the sense that whenever it is satisfied by all members of a family of relations it is also satisfied by their union) to be \emph{strongly first order}, in the sense that the logic obtained by adding them to First Order Logic is no more expressive than First Order Logic itself. 
\end{abstract}
\section{Introduction}
Team Semantics is a generalization of Tarskian Semantics in which formulas are satisfied or not satisfied by sets of assignments (called \emph{teams}) rather than by single assignments. This semantics was developed by Hodges in \cite{hodges97} as a compositional alternative to the game-theoretic semantics for Independence-Friendly Logic \cite{hintikkasandu89}, an extension of first order logic (related to \emph{Branching Quantifier Logic} \cite{henkin61}) that adds to it \emph{slashed quantifiers} $\exists x / \tuple y$, with the meaning of ``there exists some $x$, chosen independently from all variables in $\tuple y$, such that\ldots''. 

As observed by Jouko V\"a\"an\"anen, in Team Semantics this notion of dependence/independence between variables corresponds precisely to database-theoretic \emph{functional dependence} \cite{armstrong74}. This observation led to the development of \emph{Dependence Logic} \cite{vaananen07}, which separates the notion of dependency restriction from the notion of quantification by means of (functional) \emph{dependence atoms} $=\!\!(\tuple y, x)$ meaning ``the value of $x$ is \emph{determined by} the values of $\tuple y$''; and, later, to the development of an entire family of novel extensions of First Order Logic likewise obtained by adding to it new atoms to express forms of dependence/independence between variables. 

The properties of some of the logics thus generated are, by now, fairly well studied; but much is still not known about the general classification of the logics obtainable from First Order Logic in this manner. In particular, no full answer yet exists to the following question: which dependencies between variables, if added to First Order Logic as new atoms, generate a logic that is more expressive than First Order Logic, and which ones instead fail to do so?  Aside from being a natural starting point for a general classification of these extensions of First Order Logic, this question is of philosophical interest: after a fashion, a full answer to it would clarify how much we can get away with adding to First Order Logic (in terms of novel atoms specifying dependencies between variables) before the resulting logic becomes unrecoverably second-order in nature. 

This paper continues a series of works \cite{galliani2015upwards,galliani2016strongly,galliani2020safe,galliani2019characterizing,galliani2019nonjumping,galliani2021doubly} that attempt to solve this problem with respect to restricted classes of dependencies, as part of a program to eventually reach a general characterization. In particular, here we will focus on dependencies that have the two properties of being \emph{domain-independent} and \emph{union-closed}, which are shared by many of the dependencies of interest in the literature, and provide a full solution to the problem for their case.
\section{Preliminaries}
\label{sect:prelims}
\begin{definition}[Team]
Let $\M$ be a first order model with domain $M$ and let $V$ be a set of variables. Then a team over $\M$ with domain $V$ is a set of assignments $s: V \rightarrow M$. Given such a team $X$ and some tuple of variables $\tuple v = (v_1 \ldots v_k) \in V^k$, we will write $X(\tuple v)$ for the $|\tuple v|$-ary relation $\{s(\tuple v) : s \in X\} \subseteq M^{|v|}$, where $s(\tuple v)$ is the tuple $(s(v_1) \ldots s(v_k))$.  
\end{definition}
\begin{definition}[Team Semantics for First Order Logic]
	Let $\M$ be a first order model with domain $M$, let $\phi$ be a first order formula in Negation Normal Form\footnote{In this work we will assume that all expressions are in Negation Normal Form. Also, we will use the usual abbreviations of First Order Logic, such as $\tuple x = \tuple y$ for $\bigwedge_i (x_i = y_i)$ where $\tuple x = (x_1 \ldots x_k)$ and $\tuple y = (y_1 \ldots y_k)$ are tuples of variables of the same length.} over the signature of $\M$, and let $X$ be a team over $\M$ whose domain contains the free variables of $\phi$. Then we say that $X$ satisfies $\phi$ in $\M$, and we write $\M \models_X \phi$, if this follows from the following rules:
\begin{description}
    \item[TS-lit:] If $\phi$ is a first order literal, $\M \models_X \phi$ if and only if, for all assignments $s \in X$, $\M \models_s \phi$ in the usual sense of Tarskian Semantics; 
    \item[TS-$\vee$:] $\M \models_X \phi_1 \vee \phi_2$ if and only if $X = Y \cup Z$ for two $Y, Z \subseteq X$ such that $\M \models_Y \phi_1$ and $\M \models_Z \phi_2$;
    \item[TS-$\wedge$:] $\M \models_X \phi_1 \wedge \phi_2$ if and only if $\M \models_X \phi_1$ and $\M \models_X \phi_2$; 
    \item[TS-$\exists$:] $\M \models_X \exists v \psi$ if and only if there exists some function\footnote{$\mathcal P(M)$ represents the powerset $\{A : A \subseteq M\}$ of $M$.} $H: X \rightarrow \mathcal P(M) \backslash \{\emptyset\}$ such that $\M \models_{X[H/v]} \psi$, where $X[H/v] = \{s[m/v] : s \in X, m \in H(s)\}$;
    \item[TS-$\forall$:] $\M \models_X \forall v \psi$ if and only if $\M \models_{X[M/v]} \psi$, where $X[M/v] = \{s[m/v] : s \in X, m \in M\}$.
\end{description}
    A sentence $\phi$ is true in a model $\M$ if and only if $\M \models_{\{\epsilon\}} \phi$, where $\{\epsilon\}$ is the team containing only the empty assignment $\epsilon$. In that case, we write that $\M \models \phi$. 
\end{definition}
Over First Order Logic, Team Semantics reduces to Tarskian Semantics:
\begin{proposition}(\cite{vaananen07}, Corollary 3.32)
Let $\M$ be a first order model, let $\phi$ be a first order formula in Negation Normal Form over the signature of $\M$, and let $X$ be a team over $\M$ whose domain contains the free variables of $\phi$. Then $\M \models_X \phi$ if and only if, for all $s \in X$, $\M \models_s \phi$ in the sense of Tarskian Semantics. 

In particular, if $\phi$ is a first order sentence, $\M \models \phi$ in the sense of Team Semantics if and only if $\M \models \phi$ in the sense of Tarskian Semantics. 
	\label{propo:flat}
\end{proposition}
However, the framework of Team Semantics allows one to extend First Order Logic in novel ways, for example via \emph{functional dependence atoms} \cite{vaananen07}, \emph{inclusion and exclusion atoms} \cite{galliani12}, \emph{anonymity atoms}  \cite{galliani12c,ronnholm2018arity,vaananenAnon}\footnote{Anonymity atoms were called "non-dependence atoms" and briefly studied in Section 4.6 of \cite{galliani12c}, in which their equidefinability with inclusion atoms was also proved; but their interpretation and their properties are discussed more in depth in \cite{ronnholm2018arity,vaananenAnon}.} or \emph{independence atoms} \cite{gradel13}
\begin{description}
\item[TS-$=\!\!(\cdot;\cdot)$:] If $\tuple x$ and $\tuple y$ are tuples of variables in the domain of $X$, $\M \models_X =\!\!(\tuple x; \tuple y)$ if and only if, for all $s, s' \in X$, $s(\tuple x) = s'(\tuple x) \Rightarrow s(\tuple y) = s'(\tuple y)$;
\item[TS-$\subseteq$:] If $\tuple x$ and $\tuple y$ are tuple of variables in the domain of $X$ of the same length, $\M \models_X \tuple x \subseteq \tuple y$ if and only if $X(\tuple x) \subseteq X(\tuple y)$;
\item[TS-$|$:] If $\tuple x$ and $\tuple y$ are tuple of variables in the domain of $X$ of the same length, $\M \models_X \tuple x \mid \tuple y$ if and only if $X(\tuple x) \cap X(\tuple y) = \emptyset$;
\item[TS-\anon:] If $\tuple x$ and $\tuple y$ are tuples of variables in the domain of $X$, $\M \models_X \tuple x \anon \tuple y$ if and only if for every $s \in X$ there exists some $s' \in X$ with $s(\tuple x) = s'(\tuple x)$ but $s(\tuple y) \not = s'(\tuple y)$
\item[TS-$\bot$:] If $\tuple x$ and $\tuple y$ are tuples of variables in the domain of $X$, $\M \models_X \tuple x \bot \tuple y$ if and only if $X(\tuple x \tuple y) = X(\tuple x) \times X(\tuple y)$, i.e., for all $s, s' \in X$ there is some $s'' \in X$ with $s''(\tuple x) = s(\tuple x)$ and $s''(\tuple y) = s'(\tuple y)$. 
\end{description}
The logics obtained by adding these dependencies to First Order Logic $\FO$ are all more expressive than it: in particular, (Functional) Dependence Logic $\FO(=\!\!(\cdot;\cdot))$, Exclusion Logic $\FO(|)$ and Independence Logic $\FO(\bot)$ are all equivalent to\footnote{In the sense that every sentence of any of these logics is equivalent to some Existential Second Order sentence, and vice versa. As we will see, it is however not the case that every \emph{formula} of $\FO(\bot)$ is equivalent to some formula of $\FO(|)$.} Existential Second Order Logic \cite{vaananen07,galliani12,gradel13,vaananen13}, while Inclusion Logic $\FO(\subseteq)$ and Anonimity Logic $\FO(\anon)$ are both equivalent to the positive fragment of Greatest Fixpoint Logic \cite{gallhella13,galliani12c}. On the other hand, the  \emph{nonemptiness atoms}
\begin{description}
\item[TS-$\NE$:] $\M \models_X \NE (\tuple v)$ if and only if\footnote{The choice of $\tuple v$ in the atom $\NE(\tuple v)$ is indifferent as long as it is contained in the domain of $X$. Thus, one could define $\NE$ as a single $0$-ary dependency instead of a family of dependencies of all arities, as we do here for simplicity and uniformity.} $X(\tuple v) \not = \emptyset$
\end{description}
do not increase the expressive power of First Order Logic if added to it, as per Theorem 11 of \cite{galliani2015upwards},\footnote{Theorem \ref{thm:upflat} in this work.} even though it follows easily from Proposition \ref{propo:flat} that such an atom is not definable by any first order formula over Team Semantics.\footnote{Of course there is a first order sentence $\exists \tuple x R \tuple x$ that says that $R$ is not empty; however, there is no first order formula $\phi(\tuple w)$ such that $\M \models_X \phi(\tuple w)$ if and only if $X(\tuple w) \not = \emptyset$.} 

Extensions of First Order Logic with Team Semantics via additional  connectives have also been studied. One such connective that will be of some importance in this work is the \emph{Global} (or \emph{Boolean}) \emph{Disjunction} 
\begin{description}
\item[TS-$\sqcup$:] $\M \models_X \phi \sqcup \psi$ if and only if $\M \models_X \phi$ or $\M \models_X \psi$.
\end{description}

There is by now a considerable amount of research, that cannot be summarized here, concerning the classification of logics based on Team Semantics. One useful tool for such research has been the following general definition of dependency \cite{kuusisto2015double}:
\begin{definition}[Dependency]
Let $k \in \mathbb N$. Then a $k$-ary dependency $\D$ is a class, closed under isomorphisms, of models $\A= (A, R)$, where $R$ is a $k$-ary relation over $A$. Given a first order model $\M$ with domain $M$, a team $X$ over it and a tuple $\tuple x$ of variables in the domain of $X$ such that $|\tuple x| = k$, 
\begin{description}
\item[TS-D:] $\M \models_X \D \tuple x$ if and only if $(M, X(\tuple x)) \in \D$. 
\end{description}
Given any such dependency $\D$, we write $\FO(\D)$ for the logic obtained by adding $\D$ to First Order Logic with Team Semantics; and likewise, if $\DD$ is a family of such dependencies we write $\FO(\DD)$ for the logic obtained by adding all dependencies in $\DD$ to First Order Logic with Team Semantics. 
\end{definition}

This definition is somewhat more general than the examples given so far in that it allows an instance $\D \tuple x$ of a dependency $\D$ to be satisfied or not satisfied by a team $X$ over a model $\M$ depending on the existence in $\M$ of elements that do not occur in $X(\tuple x)$. In other words, all dependencies mentioned so far (and nearly all other dependencies considered in the literature) are \emph{domain independent} in the sense of the following definition from \cite{kontinen2016decidable}:\footnote{In \cite{kontinen2016decidable}, this property is called "Universe Independence".}
\begin{definition}[Domain Independence]
A $k$-ary dependency $\D$ is \emph{domain independent} if and only if, for all models $\A = (A, R)$, 
\[
    (A, R) \in \D \Leftrightarrow (\Fld(R), R) \in \D
\]
where $\Fld(R) = \bigcup_{i=1}^k \{\tuple a_i : \tuple a \in R\}$ is the set of all elements that occur in an position of any tuple of $R$. 
\end{definition}
An example of a dependency that is not domain-independent is given by the \emph{Totality Atoms} from \cite{galliani2015upwards}
\begin{description}
\item[TS-$\All$:] $\M \models_X \All (\tuple x)$ if and only if $X(\tuple x)= M^{|\tuple x|}$
\end{description}
that state that a certain tuple of variables $\tuple x$ takes \emph{all} possible values in the current team. Domain independence is a very natural condition that is satisfied by the majority of the dependencies studied in the literature, and it can be argued that it should be included in the definition of a "proper" dependency. 

All dependencies mentioned so far, including totality, are also \emph{First Order} in the sense that, if seen as properties of relations, they are first order definable: 
\begin{definition}[First Order Dependency]
A $k$-ary dependency $\D$ is First Order if and only if there exists a first order sentence $\D(R)$ such that 
\[
    (A, R) \in \D \Leftrightarrow (A, R) \models \D(R)
\]
for all models $\A = (A, R)$, where $R$ is a $k$-ary relation.
\label{def:fodep}
\end{definition}
Ultimately, the fact that adding first order dependencies to First Order Logic may lead to logics more expressive than First Order Logic itself is a consequence of the second-order nature of the Team Semantics rules $\textbf{TS-$\vee$}$ and $\textbf{TS-$\exists$}$ for disjunction and existential quantification.\footnote{The \textbf{FOT} logic of \cite{kontinen2019logics}, which is a Team Semantics-based logic with no  higher order quantifications in the rules for its connectives, is no stronger than First Order Logic and can define all first order dependencies.}

There are some general classes of dependencies that have also been useful in the study of Team Semantics: 
\begin{definition}[Empty Team Property, Downwards, Union- and Upwards Closures]

Let $\D$ be a $k$-ary dependency. Then $\D$
\begin{itemize}
    \item has the \emph{Empty Team Property} if $(A, \emptyset) \in \D$ for all domains $A$;
    \item is \emph{Downwards Closed} if whenever $(A, R) \in \D$ and $S \subseteq R$ we also have that $(A, S) \in \D$; 
    \item is \emph{Union Closed} if whenever $(R_i)_{i \in I}$ is a family of relations over $A$ such that $(A, R_i) \in \D$ for all $i \in I$ then $\left(A, \bigcup_i R_i\right) \in \D$ as well; 
    \item is \emph{Upwards Closed} if whenever $(A, R) \in \D$ and $R \subseteq S$ then $(A, S) \in \D$. 
\end{itemize}
\end{definition}
Note that union closure implies the empty team property, because for $I = \emptyset$ we have that $\bigcup_{i \in I} R_i = \emptyset$.

Table \ref{tab:tab1} summarizes the properties of the dependencies mentioned so far.\footnote{$\NE$ and $\All$ are not union closed because the union of the empty family of relations is the empty relation, which does not satisfy these dependencies. If we restricted the definition of union closure to nonempty families $I$, they would be union closed.}.
\begin{table}[]
    \centering
        \caption{Closure Properties of Dependencies}
    \begin{tabular}{c|c c c c c}
         Dependency & Empty Team & ~~~Downwards~~~ & ~~~Union~~~ & ~~~Upwards~~~ & ~~~Domain Ind.~~~\\
         \hline 
         $=\!\!(\cdot; \cdot)$ & $+$& $+$ & $-$  & $-$ & $+$\\
         $\subseteq$ & $+$ & $-$ & $+$ & $-$ & $+$\\
         $|$ & $+$ & $+$ & $-$ & $-$ & $+$\\
         $\anon$ & $+$ & $-$ & $+$ & $-$ & $+$\\
         $\bot$ & $+$ & $-$ & $-$ & $-$ & $+$\\
         $\NE$ & $-$ & $-$ & $-$ & $+$ & $+$\\
         $\All$ & $-$ & $-$ & $-$ & $+$ & $-$
    \end{tabular}
    \label{tab:tab1}
\end{table}

It can be seen, by structural induction on formulas, that if $\DD$ is a family of downwards closed dependencies then all formulas $\phi \in \FO(\DD)$ are downwards closed, in the sense that if $\M \models_X \phi$ and $Y \subseteq X$ then $\M \models_Y \phi$. Likewise, if all dependencies in $\DD$ are union closed and $\M \models_{X_i} \phi \in \FO(\DD)$ for all $i \in I$, where $(X_i)_{i \in I}$ is a family of teams over $\M$ with the same domain, then $\M \models_{\bigcup_{i \in I} X_i} \phi$; and if all dependencies in $\DD$ have the empty team property then $\M \models_\emptyset \phi$ for all models $\M$ and all formulas $\phi \in \FO(\DD)$ over the signature of $\M$. This implies at once that, for example, independence atoms are not definable in terms of exclusion atoms: even though every $\FO(\bot)$ sentence is equivalent to some $\FO(|)$ sentence and vice versa, there exists no $\FO(|)$ formula that is equivalent to the independence atom $\tuple x \bot \tuple y$. This discrepancy between definability and expressivity in Team Semantics -- i.e., the fact that logics based on Team Semantics may be equiexpressive despite the corresponding atoms not being reciprocally definable -- is one of the most intriguing phenomena of Team Semantics. 

The property of upwards closure, instead, is not preserved by Team Semantics as first order literals are not upwards closed. However,
\begin{theorem}[\cite{galliani2015upwards}, Theorem 11]
Let $\DD$ be a family of upwards closed, first order dependencies. Then every sentence of $\FO(\DD)$ is equivalent to some first order sentence. 
\label{thm:upflat}
\end{theorem}
In other words, the set $\DD^\uparrow$ of all upwards closed first order dependencies is \emph{strongly first order} in the sense of the following definition: 
\begin{definition}[Strongly First Order Dependencies]
Let $\D$ be a dependency [let $\DD$ be a family of dependencies]. Then $\D$ [$\DD$] is strongly first order if and only if every sentence of $\FO(\D)$ [$\FO(\DD)$] is equivalent to some first order sentence.  
\end{definition}
It is easy to see that all strongly first order dependencies are first order in the sense of Definition \ref{def:fodep}, since by the rules of Team Semantics $(M, R) \in \D$ if and only if $(M, R) \models \forall \tuple x (\lnot R \tuple x \vee (R \tuple x \wedge \D \tuple x))$ and if $\D$ is strongly first order this is equivalent to some first order sentence; but this is certainly no sufficient condition, since as we saw there exist first order dependencies that are not strongly first order. 

On the other hand, being first order and upwards closed is a sufficient condition for a dependency to be strongly first order, but it is not necessary: for example, as shown already in Section 3.2 of \cite{galliani12}, \emph{constancy atoms}
\begin{description}
\item[TS-$=\!\!(\cdot)$:] $\M \models_X =\!\!(\tuple x)$ if and only if for all $s, s' \in X$, $s(\tuple x) = s'(\tuple x)$,
\end{description}
that can be seen as degenerate cases of functional dependencies stating that $\tuple x$ depends on \emph{nothing} (i.e. its value is determined by the empty tuple of variables), are strongly first order despite not being upwards closed.\footnote{We could consider only \emph{unary} constancy atoms, as $=\!\!(\tuple x) \equiv \bigwedge_{x \in \tuple x} =\!\!(x)$; but, similarly to the $\NE$ case, for simplicity and uniformity we define constancy atoms for all arities.} We can say more: 
\begin{proposition}
Every sentence of $\FO(=\!\!(\cdot), \DD^\uparrow, \sqcup)$ -- i.e. of First Order Logic augmented with constancy atoms, upwards closed first order dependencies, and global disjunctions -- is equivalent to some first order sentence. 
\label{propo:up_const_global}
\end{proposition}
\begin{proof}
By Theorem 21 of \cite{galliani2015upwards}, $\DD^\uparrow \cup =\!\!(\cdot)$ is strongly first order. By Proposition 14 of \cite{galliani2019nonjumping}, if $\DD$ is a strongly first order family of dependencies, every sentence of $\FO(\DD, \sqcup)$ is equivalent to some first order sentence. The result follows at once. 
\end{proof}
These observations led to the following 

\noindent \textbf{Problem:} Provide necessary and sufficient conditions for a dependency or a family of dependencies to be strongly first order. 

and to the related 

\noindent \textbf{Conjecture:} A family of dependencies $\DD$ is strongly first order if and only if all $\D \in \DD$ are strongly first order. 

A partial answer was found in \cite{galliani2019characterizing}:
\begin{theorem}[\cite{galliani2019characterizing}, Theorem 4.5]
Let $\D$ be a downwards closed, domain independent\footnote{The result of \cite{galliani2019characterizing} actually requires a weaker condition than domain independence called \emph{relativizability}. This property is implied by domain independence, and in this work we will focus on domain independent dependencies. Totality $\All$ is an example of a dependency that is relativizable but not domain independent.} dependency with the empty team property. Then $\D$ is strongly first order if and only if it is definable in Constancy Logic $\FO(=\!\!(\cdot))$
\end{theorem}

In \cite{galliani2019nonjumping}, a characterization was found for the strongly first order dependencies that are domain independent\footnote{Again, relativizability suffices for this result.} and have the (somewhat technical) property of being \emph{non-jumping}, which in particular is true of strongly first order downwards closed dependencies (even of those without the empty team property): 
\begin{theorem}[\cite{galliani2019nonjumping}, Corollary 31]
Let $\D$ be a non-jumping (or, in particular, downwards closed), domain independent\footnote{Or relativizable.}, strongly first order dependency. Then $\D$ is definable in $\FO(=\!\!(\cdot), \DD^\uparrow, \sqcup)$. 
\label{thm:nonjumping}
\end{theorem}

Finally, in \cite{galliani2021doubly}, a characterization was found for dependencies that are domain-independent\footnote{Again, or just relativizable.} and \emph{doubly} strongly first order, in the sense that they and their negation can be added to First Order Logic (separately or jointly) without increasing its expressive power. 

In the next section we will find a necessary and sufficient condition for \emph{union-closed}, domain-independent dependencies to be relativizable, thus completing the characterization of ``strong first-orderness'' with respect to dependencies satisfying the main closure properties discussed in the literature, modulo the condition of domain independence (which, as mentioned, is a natural condition that could be argued to apply to any sensible dependency notion).\footnote{Nonetheless, there are dependencies of interest, such as independence atoms $\bot$, which satisfy none of these closure properties; and thus, the problem of fully characterizing strongly first order dependencies remains not entirely solved yet.}
\section{Characterizing Union-Closed Dependencies}
A key insight for our characterization is the observation that if $\D$ is domain-independent and strongly first order then there cannot exist a chain of models satisfying and not satisfying $\D$ as per Figure \ref{fig:forbidden_chain}:

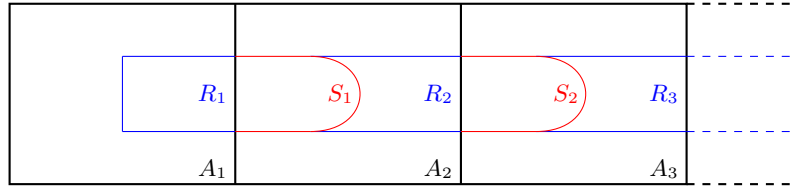
\begin{figure}
		\begin{center}
		\begin{tikzpicture}
		     \draw[black, thick] (0,-1.2) rectangle (9,1.2);
		     \draw[black, thick,dashed] (9, 1.2) -- (10.5, 1.2);
		     \draw[black, thick, dashed] (9, -1.2) -- (10.5, -1.2);
		     \draw[black, thick] (3,1.2)--(3,-1.2);
		     \draw[black, thick](6,1.2)--(6,-1.2);
		     
		     \draw[blue] (1.5,0.5) -- (3,0.5);
		     \draw[blue] (4,0.5) -- (6,0.5);
		     \draw[blue] (7,0.5) -- (9,0.5);
		     
		     \draw[blue] (1.5,0.5) -- (1.5,-0.5);
		     
		     \draw[blue] (1.5,-0.5) -- (3,-0.5);
		     \draw[blue] (4,-0.5) -- (6,-0.5);
		     \draw[blue] (7,-0.5) -- (9,-0.5);
		     
		     \draw[blue, dashed] (9,-0.5) -- (10.5,-0.5);
		     \draw[blue, dashed] (9,0.5) -- (10.5,0.5);
		     
		     \draw[red] (4,0.5) to[out=-0,in=0,distance=25] (4,-0.5);
		     \draw[red] (4,0.5) -- (3,0.5);
		     \draw[red] (4,-0.5) -- (3, -0.5); 
		     
		     \draw[red] (7,0.5) to[out=-0,in=0,distance=25] (7,-0.5);
	        \draw[red] (7,0.5) -- (6,0.5);
		     \draw[red] (7,-0.5) -- (6, -0.5); 
		     
		     \node at (2.7,-1){$A_1$};
		     \node at (5.7,-1){$A_2$};
		     \node at (8.7,-1){$A_3$};
		     \node[blue] at (2.7,0){$R_1$};
		     \node[blue] at (5.7,0){$R_2$};
		     \node[blue] at (8.7,0){$R_3$};
		     \node[red] at (4.4,0){$S_1$};
		     \node[red] at (7.4,0){$S_2$};
		\end{tikzpicture}
			 \end{center}
				\caption{
				A forbidden elementary chain. As per Lemma \ref{lemma:nochain}, if $\D$ is domain independent and strongly first order, it cannot be the case that, for all $i \in \mathbb N$, $(A_i, R_i) \preceq (A_{i+1}, R_{i+1})$, $(A_i, R_i) \in \D$, and $(A_{i+1}, S_i) \not \in \D$ for $R_i \subseteq S_i \subseteq R_{i+1}$.
				}
		 \label{fig:forbidden_chain}
	 \end{figure}
\begin{lemma}
Let $\D$ be a $k$-ary domain-independent dependency and suppose that there exists an elementary chain $(\A_i)_i \in \mathbb N$, where each $\A_i$ is of the form $(A_i, R_i)$ for some domain $A_i$ and some $k$-ary relation $R_i$ over $A_i$, such that 
\begin{itemize}
    \item $\A_i = (A_i, R_i) \in \D$ for all $i \in \mathbb N$; 
    \item $\A_i$ is an elementary substructure of $\A_{i+1}$ ($\A_i \preceq \A_{i+1}$) for all $i \in \mathbb N$; 
    \item For all $i \in \mathbb N$ there exists some $k$-ary relation $S_i$ over $A_{i+1}$ such that $R_i \subseteq S_i \subseteq R_{i+1}$ and $(A_{i+1}, S_i) \not \in \D$.
\end{itemize}
Then $\D$ is not strongly first order. 
\label{lemma:nochain}
\end{lemma}
\begin{proof}
By the L\"owenheim-Skolem Theorem, we can assume that all $(A_i, R_i)$ are countable. Then let $\A = (A, R) = \left( \bigcup_i A_i, \bigcup_i R_i\right)$ be the union of the elementary chain. By the Elementary Chain Theorem we have that $\A \succeq \A_i$ for all $i \in \mathbb N$, and therefore in particular $(A, R) \models \D(R)$; and furthermore since $A$ is the countable union of countable sets it is countable and we can identify it up to isomorphism with the set of natural numbers $\mathbb N$. 

Then let $\B$ be a model, with the same domain $\mathbb N$, whose signature $\Sigma$ contains the binary predicate $<$, the unary function $\mathfrak s$, and the $(k+1)$-ary predicate $\hat{S}$, such that $<$ is interpreted as the usual ordering over $\mathbb N$, $\mathfrak s$ is interpreted as the corresponding successor function $n \mapsto n+1$, and $\hat{S}$ is interpreted as the relation $\{(i, \tuple{m}) : \tuple m \in S_i\}$. 

I state that if $\D$ is strongly first order and domain independent then $\B$ has no uncountable elementary extensions. This is impossible by the L\"owenheim-Skolem Theorem; therefore, if I can prove this I can conclude that $\D$ cannot be strongly first order. 

If the dependency $\D$ is strongly first order and domain-independent, the complete first order theory of $\B$ contains the sentence $\forall v \lnot \phi(v)$, where $\phi(d)$ is the first order sentence over the signature $\Sigma \cup \{d\}$ that is equivalent to the $\FO(\D)$ sentence 
\[
    \exists t (t < d \wedge \forall \tuple w(\lnot \hat S t \tuple w \vee (\hat S t \tuple w \wedge \D(\tuple w))))
\]
that, as we will later verify, states that there exists some nonempty set of indices $I \subseteq \{1 \ldots d\}$ such that $\left(\mathbb N, \bigcup_{i \in I}S_i\right) \in \D$. Indeed, for any such set we have that $\bigcup_{i \in I}S_i = S_q$ for $q = \max(I)$, and by domain independence $(\mathbb N, S_q) \not \in \D$ since $(A_{q+1}, S_q) \not \in \D$. 

Now let $\B' $ be an uncountable elementary extension of $\B$ with domain $B' \supsetneq \mathbb N$, let $d \in B' \backslash \mathbb N$ be a non-standard element of it, and choose $I = \mathbb N$. Then $i < d$ for all $i \in \mathbb N$, and furthermore since $\bigcup_{i \in \mathbb N} S_i = \bigcup_{i \in \mathbb N} R_i$ and since as we already observed $(\bigcup_i A_i, \bigcup_i R_i) \in \D$ we have that $(\bigcup_i A_i, \bigcup_i S_i) \in \D$. But then, as $\D$ is domain independent, $(B', \bigcup_i S_i) \in \D$. Therefore $\B' \models \phi(d)$, which contradicts the claim that $\B'$ is an elementary extension of $\B$. \vskip 1em

It remains to check that the formula $\phi(d)$ indeed specifies the required property. Suppose that $\M \models \phi(d)$. Then, by the rules of Team Semantics, there exists some function $H: \{\epsilon\} \rightarrow \mathcal P(M)\backslash \{\emptyset\}$ such that $\M \models_{\{\epsilon\}[H/t]} t<d \wedge \forall \tuple w(\lnot \hat S t \tuple w \vee (\hat S t \tuple w \wedge \D(\tuple w)))$. Now let $X = \{\epsilon\}[H/t]$ and let $I = X(t)$. 

Then $I \not = \emptyset$; and since $\M \models_X t < d$, $i < d$ for all $i \in I$, as required. 

For $Y = X[M/\tuple w]$,\footnote{This is a shorthand for $X[M/\tuple w_1][M/\tuple w_2] \ldots [M/\tuple w_k]$.} we furthermore have that $\M \models_Y \lnot \hat S t \tuple w \vee (\hat S t \tuple w \wedge \D(\tuple w))$. Therefore, $Y = Y_1 \cup Y_2$, where 
\begin{itemize}
    \item For all $s \in Y_1$, $s(\tuple w) \not \in S_{s(t)}$; 
    \item For all $s \in Y_2$, $s(\tuple w) \in S_{s(t)}$; 
    \item $(M, Y_2(\tuple w)) \in \D$. 
\end{itemize}
Then necessarily $Y_1 = \{s \in Y: s(\tuple w) \not \in S_{s(t)}\}$ and $Y_2 = \{s \in Y: s(\tuple w) \in S_{s(t)}\}$; but by construction, in $Y$ the tuple $\tuple w$ takes all possible values in $M^k$ for all possible values $i \in I$ of $t$, and therefore $Y_2(\tuple w) = \{\tuple m : \exists i \in I \text{ s.t. } \tuple m \in S_i\} =  \bigcup_{i \in I} S_i$. Therefore $(M, \bigcup_{i \in I} S_i) \in \D$ as required. 

Conversely, suppose that there exists a nonempty set $I$ of elements less than $d$ such that $(M, \bigcup_{i \in I} S_i) \in \D$. Then let $H: \{\epsilon\} \rightarrow \mathcal P(M) \backslash \{\emptyset\}$ be such that $H(\epsilon) = I$, and let $X = \{\epsilon\}[H/t]$ so that $X(t) = I$. 

Clearly $\M \models_X t < d$, and I only need to prove that -- for $Y = X[M/\tuple w]$ -- $\M \models_Y \lnot \hat S t \tuple w \vee (\hat S t \tuple w \wedge \D(\tuple w))$. Now let $Y_1 = \{s \in Y : s(\tuple w) \not \in S_{s(t)}\}$ and $Y_2 = \{s \in Y : s(\tuple w) \in S_{s(t)}\}$: then $Y = Y_1 \cup Y_2$, $\M \models_{Y_1} \lnot \hat S t \tuple w$ and $\M \models_{Y_2} \hat S t \tuple w$. Moreover, $Y_2(\tuple w) = \{s(\tuple w) \in X[M/\tuple w]: s(\tuple w) \in S_{s(t)}\} = \bigcup_{i \in X(t)} S_i = \bigcup_{i \in I} S_i$, and therefore $(M, Y_2(\tuple w)) \in \D$ and $\M \models_{Y_2} \D\tuple w$ as required.
\end{proof}
In fact, if $\D$ is strongly first order and domain independent even one single "step" (as depicted in Figure \ref{fig:forbidden_step}) of the chain of Figure \ref{fig:forbidden_chain} cannot occur:
 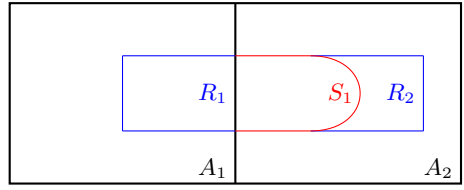
\begin{figure}
	     \begin{center}
		\begin{tikzpicture}
		     \draw[black, thick] (0,-1.2) rectangle (6,1.2);
		     \draw[black, thick] (3,1.2)--(3,-1.2);

		     \draw[blue] (1.5,0.5) -- (3,0.5);
		     \draw[blue] (4,0.5) -- (5.5,0.5);
		     
		     \draw[blue] (1.5,0.5) -- (1.5,-0.5);
		     \draw[blue] (1.5,-0.5) -- (3,-0.5);
		     \draw[blue] (4,-0.5) -- (5.5,-0.5);
		     
		     \draw[blue] (5.5,0.5) to (5.5,-0.5);
		     
		     \draw[red] (4,0.5) to[out=-0,in=0,distance=25] (4,-0.5);
		     \draw[red] (4,0.5)--(3,0.5);
		     \draw[red] (4,-0.5)--(3,-0.5);
		     
		     \node at (2.7,-1){$A_1$};
		     \node at (5.7,-1){$A_2$};
		     \node[blue] at (2.7,0){$R_1$};
		     \node[blue] at (5.2,0){$R_2$};
		     \node[red] at (4.4,0){$S_1$};
		\end{tikzpicture}
		\end{center}
			\caption{
				A forbidden elementary "step". As per Proposition \ref{propo:nostep}, if $\D$ is domain independent and strongly first order, it cannot be the case that $(A_1, R_1) \preceq (A_2, R_2)$, $(A_1,R_1) \in \D$, $(A_2, R_2) \in \D$, and $(A_2, S_1) \not \in \D$ for some $S_1$ such that $R_1 \subseteq S_1 \subseteq R_2$.
				}
		 \label{fig:forbidden_step}
	 \end{figure}
\begin{proposition}
Let $\D$ be a domain independent, strongly first order dependency and suppose that $(A_1, R_1) \in \D$ and that $(A_1, R_1) \preceq (A_2, R_2)$. Then for all relations $S_1$ with $R_1 \subseteq S_1 \subseteq R_2$, $(A_2, S_1) \in \D$ as well.  
\label{propo:nostep}
\end{proposition}
\begin{proof}
Suppose that this is not the case, i.e. that -- as per Figure \ref{fig:forbidden_step} -- $(A_1, R_1) \preceq (A_2, R_2)$, $R_1 \subseteq S_1 \subseteq R_2$ and $(A_2, S_1) \not \in \D$. 

Then we can build a chain as per Figure \ref{fig:forbidden_chain} as follows: start by with the given $A_1$, $R_1$, $A_2$, $R_2$, and $S_1$. Then consider the theory 
\begin{align*}
    &T := \{\phi(R_3, \tuple a): \tuple a \subseteq A_2, (A_2, R_2) \models \phi(R_2, \tuple a), \phi \in \FO\} \cup \{S_2 \tuple b : \tuple b \in R_2\} \cup\\
    & ~~~~~~~
    \{\forall \tuple x (S_2 \tuple x \rightarrow R_3 \tuple x), \lnot \D(S_2)\}. 
\end{align*}
By compactness, this theory is satisfiable. Indeed, any finite subset $T_0$ of $T$ is entailed by some first order sentence of the form $\phi(R_3, \tuple a) \wedge \bigwedge_{i=1}^n S_2 \tuple b_i \wedge \forall \tuple x (S_2 \tuple x \rightarrow R_3 \tuple x) \wedge \lnot \D(S_2)$, where $(A_2, R_2) \models \phi(R_2, \tuple a) \wedge \bigwedge_{i=1}^n R_2 \tuple b_i$.

But then $(A_2, R_2) \models \exists \tuple z \tuple w_1 \ldots \tuple w_n (\phi(R_2, \tuple z) \wedge \bigwedge_{i=1}^n R_2 \tuple w_i \wedge \tau(\tuple z \tuple w_1 \ldots \tuple w_n))$, where $\tau(\tuple z \tuple w_1 \ldots \tuple w_n)$ describes the identity type of $\tuple a \tuple b_1 \ldots \tuple b_n$, and so since $(A_1, R_1) \preceq (A_2, R_2)$ there exists a tuple $\tuple c$ in $A_1$ and tuples $\tuple d_1 \ldots \tuple d_n$ in $R_1$ such that $\tuple c \tuple d_1 \ldots \tuple d_n$ has the same identity type as $\tuple a \tuple b_1 \ldots \tuple b_n$ and $(A_1, R_1) \models \phi(R_1, \tuple c)$. But then $(A_2, R_2) \models \phi(R_2, \tuple c)$ as well, and moreover $\tuple d_1 \ldots \tuple d_n \in R_1 \subseteq S_1 \subseteq R_2$, and so $T_0$ can be satisfied by choosing a model with domain $A_2$ in which $\tuple a \tuple b_1 \ldots \tuple b_n$ are interpreted as $\tuple c \tuple d_1 \ldots \tuple d_n$, $S_2$ is interpreted as $S_1$, and $R_3$ is interpreted as $R_2$. 

Now, any model of $T$ is an elementary extension $(A_3, R_3)$ of $(A_2, R_2)$ with some $S_2$ with $R_2 \subseteq S_2 \subseteq R_3$ and $(A_3, S_2) \not \in \D$.\footnote{Instead $(A_3, R_3) \in \D$, because $(A_3, R_3) \succeq (A_2, R_2) \in \D$ and $\D$ is first order.} Iterating this procedure, we can find an infinite chain as per Figure \ref{fig:forbidden_chain}; but by Lemma \ref{lemma:nochain} this is impossible.
\end{proof}

Now we can find  a "local characterization" of strongly first order, domain independent, union closed dependencies:
\begin{definition}[U-Sentences]
Let $R$ be a $k$-ary relation symbol and let $\tuple a$ be a tuple of constant symbols. Then a first order sentence over the signature $\{R, \tuple a\}$ is a U-sentence if and only if it is of the form 
\[
    \exists \tuple x (\eta(\tuple x) \wedge \forall \tuple y (R \tuple y \rightarrow \theta(\tuple x, \tuple y)))
\]
where $\tuple x$ and $\tuple y$ are disjoint tuples of variables, $\eta(x)$ is a conjunction of first order literals over the signature $\{R, \tuple a\}$ in which $R$ occurs only positively,  and  $\theta(\tuple x, \tuple y)$ is a first order formula over the signature $\{\tuple a\}$  (i.e. in which $R$ does not appear).

Given two models $\A$ and $\B$ with the same signature, we will write $\A \Rrightarrow_U \B$ if, for every $U$-sentence $\phi$, $\A \models \phi \Rightarrow \B \models \phi$. 
\end{definition}

\begin{proposition}
The conjunction of two U-sentences is equivalent to some U-sentence; and if $a$ is a constant symbol occurring on some $U$-sentence $\chi$ and $v$ is a new variable not occurring in $\chi$, $\exists v \chi[v/a]$ is also a U-sentence.\footnote{Here $\chi[v/a_{\lambda}]$ represents the formula obtained from $\chi$ by replacing the constant $a_\lambda$ with the variable $v$.}
\label{propo:Uconex}
\end{proposition}
\begin{proof}
Let $\chi =  \exists \tuple x (\eta(\tuple x) \wedge \forall \tuple y (R \tuple y \rightarrow \theta(\tuple x, \tuple y)))$
and $\chi' = \exists \tuple z (\eta'(\tuple z) \wedge \forall \tuple w (R \tuple w \rightarrow \theta'(\tuple z, \tuple w)))$
be two U-sentences over the same signature $\{R, \tuple a\}$. Renaming variables as necessary, we can assume that $\tuple x$ is disjoint from $\tuple z$ and that $\tuple y = \tuple w$. 

Then $\chi \wedge \chi'$ is equivalent to $\exists \tuple x \tuple z ( (\eta(\tuple x) \wedge \eta'(\tuple z)) \wedge \forall \tuple y (R \tuple y \rightarrow (\theta(\tuple x, \tuple y) \wedge \theta'(\tuple z, \tuple y))))$, which is also a $U$-sentence. 

Moreover, if $a$ be any constant symbol occurring in $\chi$ and $v$ is a variable that does not occur in $\chi$, $\exists v( \chi[v/a])$ is the U-sentence
\[
    \exists v \exists \tuple x (\eta(\tuple x)[v/a] \wedge \forall \tuple y (R \tuple y \rightarrow \theta(\tuple x, \tuple y)[v/a]))
\]
\end{proof}

\begin{proposition}
Let $\D$ be domain independent, strongly first order and union closed, and suppose that $(A, R) \in \D$. 

Then there exists some U-sentence $\chi(R)$ over the signature $\{R\}$ such that 
\begin{enumerate}
    \item $(A, R) \models \chi(R)$; 
    \item $\chi(R) \models \D(R)$. 
\end{enumerate}
\label{propo:local_char}
\end{proposition}
\begin{proof}
Suppose that this is not the case. Then $R$ (and, therefore, $A$) must be infinite, because otherwise there would exist a $U$-sentence $\exists \tuple x (\bigwedge_{i=1}^n R \tuple t_i \wedge \forall \tuple y (R \tuple y \rightarrow \bigvee_{i} \tuple y = \tuple t_i))$ that fixes $R$ and thus - by domain independence - entails $\D$.

Note that the theory $\{\chi(S) : \chi \text{ U-sentence}, (A, R) \models \chi(R)\} \cup \{\lnot \D(S)\}$ is finitely satisfiable, since otherwise some finite conjunction of U-sentences that are true of $(A, R)$ would entail $\D(R)$ and by Proposition \ref{propo:Uconex} any such finite conjunction is equivalent to some U-sentence. Therefore, by compactness, there exists a model $(B, S) \not \in \D$ such that $(A, R) \Rrightarrow_U (B, S)$. We can furthermore assume that $(B, S)$ is $|A|$-saturated. 

Now let $(a_\alpha)_{\alpha < |A|}$ enumerate all elements of $A$ and let us build a mapping $\mathfrak h: A \rightarrow B$ such that $(A, R, (a_\alpha)_{\alpha < |A|}) \Rrightarrow_U (B, S, (\mathfrak h(a_\alpha))_{\alpha < |A|})$. 

We can do so via transfinite induction on $\lambda<|A|$, by defining functions $\mathfrak h_\lambda$ each of which maps $a_{<\lambda} = (a_\alpha)_{\alpha < \lambda}$ into some $b_{< \lambda} = (b_\alpha)_{\alpha < \lambda}$ such that 
\begin{itemize}
\item $(A, R, a_{<\lambda}) \Rrightarrow_U (B, S, b_{<\lambda})$;
\item If $\lambda < \lambda'$ then $\mathfrak h_{\lambda} \subseteq \mathfrak h_{\lambda'}$.
\end{itemize}
Then by taking $\mathfrak h = \bigcup_{\lambda < |A|} \mathfrak h_\lambda$ we will obtain the required mapping. 

\begin{enumerate}
    \item For $\lambda = 0$, we already know that $(A, R) \Rrightarrow_U (B, S)$. 
    \item Suppose that $\lambda + 1$ is a successor ordinal and that, for all $\gamma < \lambda + 1$, mappings $\mathfrak h_\gamma$ exist with the required properties. Then in particular $\mathfrak h_{\lambda}$ exists and maps $\tuple a_{<\lambda}$ into some $\tuple b_{<\lambda}$ such that $(A, R, \tuple a_{<\lambda}) \Rrightarrow_U (B, S, \tuple b_{<\lambda})$. Now consider the element $a_{\lambda}$ and the theory 
    \[
        \Theta(R, \tuple a_{<\lambda}, a_{\lambda}) = \{\chi : \chi \text{ is a U-sentence}, (A, R, \tuple a_{<\lambda}, a_{\lambda}) \models \chi\}
    \]
    Now, if $\Theta_0$ is any finite subset of $\Theta$, by Proposition \ref{propo:Uconex} the expression $\xi_0 := \exists v \bigwedge_{\chi \in \Theta_0} \chi[v/a_{\lambda}]$ is equivalent to some $U$-sentence over $\{R, \tuple a_{<\lambda}\}$ and\\$(A, R, \tuple a_{<\lambda}) \models \xi_0$; and therefore, by induction hypothesis, $(B, S, \tuple b_{<\lambda}) \models \xi_0$ and there exists some $b' \in B$ such that $(B, S, \tuple b_{<\lambda}, b') \models \chi$ for all $\chi \in \Theta_0$. Since $(B, S)$ is $|A|$-saturated and $\lambda < |A|$, this means that there exists some $b_\lambda \in B$ such that $(B, S, \tuple b_{<\lambda}, b_\lambda) \models \Theta$. But then $(A, R, \tuple a_{<\lambda+1}) \Rrightarrow_U (B, S, \tuple b_{<\lambda+1})$, and we can define $\mathfrak h_{\lambda+1} = \mathfrak h \cup \{(a_\lambda, b_\lambda)\}$. 
    \item Suppose that $\lambda$ is a limit ordinal and that, for all $\gamma < \lambda$, mappings $\mathfrak h_{\gamma}$ exist with the required properties. Then take $\mathfrak h_\lambda = \bigcup_{\gamma < \lambda} \mathfrak h_\beta$. $\\$
    
    The mapping $\mathfrak h$ thus built is in particular an embedding of $(A, R)$ into $(B, S)$. Indeed, it is injective, because expressions of the form $a_i \not = a_j$ are $U$-sentences; and furthermore, for all $k$-tuples $\tuple c$ of elements in $A$, $\tuple c \in R \Leftrightarrow (\mathfrak h(c_i))_{i=1\ldots k} \in S$  since the sentences $R \tuple c$ and $\forall \tuple y (R \tuple y \rightarrow \tuple y \not = \tuple c)$ (which is equivalent to $\lnot R \tuple c$) are also $U$-sentences. 
    
    Now, let $(A_1, R_1)$ be the image of $(A, R)$ over $\mathfrak h$. Then $(A_1, R_1)$ is isomorphic to $(A, R)$, and so in particular $(A_1, R_1) \in \D$, and furthermore $R_1 = S \cap (A_1)^k \subseteq S$. Note additionally that, for every $U$-sentence $\phi$ with constants $\tuple a$ in $A_1$, 
    $(A_1, R_1, \tuple a) \models \phi \Leftrightarrow (A, R, \mathfrak h^{-1}(\tuple a)) \models \phi \Rightarrow (B, S, \tuple a) \models \phi$
    and so every $U$-sentence with parameters in $A_1$ that holds of $(A_1, R_1)$ also holds in $(B, S)$. 
    
    Now consider any tuple $\tuple b \in S \backslash R_1$. I state that there exists some $(A_2, R_2) \succeq (A_1, R_1)$ such that $\tuple b \in R_2$, as per Figure \ref{fig:add_b}. Indeed, consider the elementary diagram $\{\phi(R_2, \tuple a): \tuple a \subseteq A_1, \phi \in \FO, (A_1, R_1, \tuple a) \models \phi(R_1, \tuple a)\}$ of $(A_1, R_1)$ together with the sentence $R_2 \tuple b$. If this theory was not satisfiable, by compactness there would exist some tuple $\tuple a$ of elements of $A_1$ and some sentence $\phi(R_1, \tuple a)$ such that $(A_1, R_1, \tuple a) \models \phi(R_1, \tuple a)$ and $\phi(R, \tuple a) \models \lnot R \tuple b$. But then, if $\tuple t$ is the tuple of terms obtained from $\tuple b$ by replacing every constant symbol not in $\tuple a$ with some new variable, we have that $\phi(R, \tuple a) \models \forall \tuple z (\lnot R \tuple t)$, where $\tuple z$ lists all newly introduced variables,\footnote{This follows from the so-called "Lemma on Constants" (e.g. Lemma 2.3.2. of \cite{hodges97b}): if $T \models \phi(\tuple c)$, where $\tuple c$ is a tuple of constants not occurring in $T$, then $T \models \forall \tuple x \phi(\tuple x)$.} and so $(A_1, R_1, \tuple a) \models \forall \tuple z (\lnot R_1 \tuple t)$ and therefore $(A_1, R_1, \tuple a) \models \forall \tuple y (R_1 \tuple y \rightarrow \forall \tuple z(\tuple y \not = \tuple t))$. But  since all $U$-sentences with parameters in $A_1$ that are true of $(A_1, R_1)$ are also true of $(B, S)$ we would then have in particular that $(B, S, \tuple a) \models \forall \tuple y (S \tuple y \rightarrow \forall \tuple z(\tuple x \not = \tuple t))$, which is impossible because $\tuple b \in S$.  Therefore the theory is satisfiable, and any model $(A_2, R_2)$ of it is an elementary extension of $(A_1, R_1)$ that contains $\tuple b$, as required.
    
    	 \begin{figure}
	     \begin{center}
		\begin{tikzpicture}
		     \draw[black, thick] (0,-1.2) rectangle (6,1.2);
		     \draw[black, thick] (3,1.2)--(3,-1.2);

		     \draw[blue] (1.5, 1.2) -- (1.5,0.2);
		     \draw[blue] (1.5,0.2) -- (5.5,0.2);
		     \draw[blue] (5.5,1.2) to (5.5,0.2);
		     
		     
		     \node at (2.7,-1){$A_1$};
		     \node at (5.7,-1){$A_2$};
		     \node[blue] at (2.7,0.7){$R_1$};
		     \node[blue] at (5.2,0.7){$R_2$};
		     
		     \node[red,circle,fill,inner sep=1pt] at (3.2, 0.4){};
		     \node[red] at (3.2, 0.7){$\tuple b$};
		     \draw[black, thick] (-1, -2) rectangle (4, 3); 
		     \node at (3.7,-1.7){$B$};
		     \draw[blue] (1.5, 1.2) -- (1.5, 2.8);
		     \draw[blue] (1.5, 2.8) to[out=0, in=90] (3.5, 0.2);
		     \node[blue] at (2.7,1.7){$S$};
            
		\end{tikzpicture}
		\end{center}
			\caption{
				$(A_1, R_1)$ is embedded in $(B, S)$, and $(A_1, R_1) \in \D$ while $(B, S) \not \in \D$, and $\tuple b \in S \backslash R_1$. Then we can build some $(A_2, R_2) \succeq (A_1, R_1)$ such that $\tuple b \in R_2$; and then, via Proposition \ref{propo:nostep}, we can show that $(A_2, R_1 \cup \{\tuple b\}) \in \D$. But then by domain independence $(B, R_1 \cup \{\tuple b\}) \in \D$ for any such $\tuple b$.
				}
		 \label{fig:add_b}
	 \end{figure}
    
    So we have that $(A_1, R_1) \preceq (A_2, R_2)$ and that $\tuple{b} \in R_2 \backslash R_1$. Then $(A_2, R_2) \in \D$, because $(A_1, R_1) \in \D$ and $(A_2, R_2)$ is elementarily equivalent to it. Therefore, $(A_2, R_1 \cup \{\tuple b\}) \in \D$: indeed, otherwise we would have that $(A_1, R_1) \in \D$, $(A_2, R_1 \cup \{\tuple b\}) \not \in \D$, $R_1 \cup \{\tuple b\} \subseteq R_2$ and $(A_2, R_2) \succeq (A_1, R_1)$, which (taking $S_1 = R_1 \cup \{\tuple b\}$) is the configuration of Figure \ref{fig:forbidden_step} that cannot occur by Proposition \ref{propo:nostep}. Then, by domain independence, $(B, R_1 \cup \{\tuple b\}) \in \D$ as well.
    
    But if this is the case for all $\tuple b \in S \backslash R_1$ and $\D$ is union closed,  \\
     $(B, \bigcup_{\tuple b \in S \backslash R_1} R_1 \cup \{\tuple b\}) = (B, S) \in \D$
     which contradicts our assumption that $(B, S) \not \in \D$. 
\end{enumerate}
\end{proof}

We can now find a global  characterization of domain independent, union-closed, strongly first order dependencies via a simple application of compactness: 
\begin{theorem}
Let $\D$ be a domain independent, union-closed, strongly first order dependency. Then $\D(R)$ is equivalent to some finite disjunction of U-sentences over the signature $\{R\}$. 
\label{thm:SFO_UC}
\end{theorem}
\begin{proof}
For each countable $(A, R) \in \D$, let $\chi_{A,R}(R)$ be a U-sentence over the signature $\{R\}$ such that $(A, R) \models \chi_{A,R}(R)$ and $\chi_{A,R}(R) \models \D(R)$. 

Then the theory $\{\lnot \chi_{A,R}(R) : (A, R) \in \D, A \text{ countable}\} \cup \{\D(R)\}$ 
has no countable models, and therefore it has no models at all; and so, by compactness, there exists a finite number of U-sentences $\chi_1 \ldots \chi_n$ such that $\D(R) \models \bigvee_{i=1}^n \chi_i(R)$. But every $\chi_i(R)$ entails $\D(R)$, and so $\D(R)$ is equivalent to $\bigvee_{i=1}^n \chi_i(R)$. 
\end{proof}
\begin{corollary}
Let $\D$ be a domain independent, union-closed dependency. Then $\D$ is strongly first order if and only if it is definable in terms of constancy atoms, nonemptiness atoms and global disjunctions. 
\label{coro:SFO_UC_def}
\end{corollary}
\begin{proof}
Suppose that $\D$ is strongly first order. Then by Theorem \ref{thm:SFO_UC}, $\D(R)$ is equivalent to some $\bigvee_{i=1}^n \chi_i(R)$ where each $\chi_i(R)$ is a $U$-sentence over the signature $\{R\}$.  Now let $\tuple w$ be a $k$-tuple of distinct variables not occurring in any $\chi_i$; and for each $\chi_i =  \exists \tuple x(\eta(\tuple x) \wedge \forall \tuple y (R \tuple y \rightarrow \theta(\tuple x, \tuple y)))$, let
\[
    \chi'_i(\tuple w) := \exists \tuple x(=\!\!(\tuple x) \wedge \eta'(\tuple x, \tuple w) \wedge \theta(\tuple x, \tuple w))
\]
where $\eta'(\tuple x, \tuple w)$ is obtained from $\eta(\tuple x)$ by replacing each conjunct of the form $R \tuple t$ with $(\top \vee (\NE(\tuple w) \wedge \tuple t = \tuple w))$.\footnote{Here $\top$ can be any literal that is always true, such as $w = w$ for some $w \in \tuple w$.}

I state that, for all models $\M$ and teams $X$ over it with domain containing $\tuple w$, $\M \models_X \D \tuple w$ if and only if $\M \models_X \bigsqcup_{i=1}^n \chi'_i(\tuple w)$.   If I can prove this then I am done: indeed, by Proposition \ref{propo:up_const_global} and by the fact that $\NE \subseteq \DD^\uparrow$ every dependency that is definable in $\FO(=\!\!(\cdot), \NE, \sqcup)$ is strongly first order. \vskip 1em

It suffices to prove that, for all models $\M$,\footnote{Note that the signature of $\M$ does not necessarily have to include the symbol $R$, which is instead interpreted as $X(\tuple w)$ when evaluating $\chi_i(R)$; and accordingly, $\chi'_i(\tuple w)$ is a $\FO(=\!\!(\cdot), \NE)$ formula over the \emph{empty} signature.} and all teams $X$ over $\M$ whose domain contains $\tuple w$, $\M \models_X \chi'_i(\tuple w)$ if and only if $(M, X(\tuple w)) \models \chi_i(R)$. 

Suppose that $\M \models_X \chi'_i(\tuple w)$: then there exists a tuple $\tuple a$ of elements such that, for $Y = X[\tuple a/\tuple x]$,\footnote{More formally: $Y = X[H_1/x_1]\ldots [H_k/x_k]$, where each $H_j$  picks the same singleton $\{a_j\}$ for all assignments in its domain -- indeed, only choice functions of this type can ensure that $\M \models_Y =\!\!(\tuple x)$.} $\M \models_Y \eta'(\tuple x, \tuple w) \wedge \theta(\tuple x, \tuple w)$. 

Then $(M, X(\tuple w)) \models \eta(\tuple a) \wedge \forall \tuple y(R \tuple y \rightarrow \theta(\tuple a, \tuple y))$. Indeed, 
\begin{enumerate}
    \item If some literal of the form $x_j = x_{j'}$ [$x_j \not = x_{j'}$] occurs among the conjuncts of $\eta$, it occurs unchanged also in $\eta'$, and since $\M \models_Y \eta$ we have that $a_j = a_{j'}$ [$a_j \not = a_{j'}$]; 
    \item If some literal of the form $R \tuple t$ occurs among the conjuncts of $\eta$, where $\tuple t$ is a tuple of variables in $\tuple x$, the expression $\top \vee (\NE(\tuple w) \wedge \tuple t = \tuple w)$ occurs among the conjuncts of $\eta'$. Therefore, $Y = Y_1 \cup Y_2$, where $Y_2 \not = \emptyset$ and $s(\tuple t) = s(\tuple w)$ for all $s \in Y_2$. But $s(\tuple t) = t[\tuple a/\tuple x]$ for all $s \in Y$. Therefore, $t[\tuple a/\tuple x] \in Y_2(\tuple w) \subseteq X(\tuple w)$. 
    \item Since $\M \models_{X[\tuple a/\tuple x]} \theta(\tuple x, \tuple w)$ and $\theta$ is first order, by Proposition \ref{propo:flat} for all $s \in X$ we have that $\M \models_s \theta(\tuple a, \tuple w)$ in the usual Tarskian sense. This means that, for all values $\tuple m$ that $\tuple w$ takes for some assignment in $X$, $\M \models \theta(\tuple a, \tuple m)$; and therefore $(M, X(\tuple w)) \models \forall \tuple y(R \tuple y \rightarrow \theta(\tuple a, \tuple y))$. 
\end{enumerate}
Therefore $(M, X(\tuple w)) \models \eta(\tuple a) \wedge \forall \tuple y(R\tuple y \rightarrow \theta(\tuple a, \tuple y))$ and hence $(M, X(\tuple w)) \models \exists \tuple x(\eta(\tuple x) \wedge \forall \tuple y(R \tuple y \rightarrow \theta(\tuple x, \tuple y)))$ as required. 

Conversely, suppose that $(M, X(\tuple w)) \models \chi_i(R)$. Then there exists a tuple $\tuple a$ of elements of $M$ such that $(M, X(\tuple w)) \models \eta(\tuple a) \wedge \forall \tuple y(R \tuple y \rightarrow \theta(\tuple a, \tuple y))$. Now let $Y = X[\tuple a/\tuple x]$: clearly, $\M \models_Y =\!\!(\tuple x)$. Moreover, $\M \models_Y \eta'(\tuple x, \tuple w) \wedge \theta(\tuple x, \tuple w)$: indeed, 
\begin{enumerate}
    \item If some literal of the form $x_j = x_{j'}$ [$x_j \not = x_{j'}$] occurs among the conjuncts of $\eta'$, the same conjunct also occurs in $\eta$. Therefore $a_j = a_{j'}$ [$a_j \not = a_{j'}$] and $\M \models_Y x_j = x_{j'}$ [$\M \models_Y x_j \not = x_{j'}$]. 
    \item If some expression of the form $\top \vee (\NE(\tuple w) \wedge \tuple t = \tuple w)$ is among the conjuncts of $\eta'$, $R \tuple t$ is among the conjuncts of $\eta$. Therefore, $\tuple t[\tuple a/\tuple x] \in X(\tuple w)$, and so there exists some $s \in Y = X[\tuple a/\tuple x]$ such that $s(\tuple w) = \tuple t[\tuple a/\tuple x] = s(\tuple t)$. But then $Y = Y \cup \{s\}$, $\M \models_Y \top$, and $\M \models_{\{s\}} \NE(\tuple w) \wedge \tuple t = \tuple w$; and therefore $\M \models_Y \top \vee (\NE(\tuple w) \wedge \tuple t = \tuple w)$, as required. 
    \item Since $(M, X(\tuple w)) \models \forall \tuple y (R\tuple y \rightarrow \theta(\tuple a, \tuple m))$, for all $\tuple m \in X(\tuple w)$ we have that $\M \models \theta(\tuple a, \tuple m)$. Therefore, for all $s \in X$, $\M \models_{s[\tuple a/\tuple x]} \theta(\tuple x, \tuple w)$ in the usual Tarskian sense; and since $Y = X[\tuple a/\tuple x]$, by Proposition \ref{propo:flat} we have that $\M \models_{Y} \theta(\tuple x, \tuple w)$. 
\end{enumerate}
Therefore, $\M \models_{X[\tuple a/\tuple x]} \eta'(\tuple x, \tuple w) \wedge \theta(\tuple x, \tuple w)$ and $\M \models_X \chi'_i(\tuple w)$. 

Thus, for every $\chi_i(R)$  we can find a $\FO(=\!\!(\cdot), \NE)$ formula $\chi'_i(\tuple w)$ such that $\M \models_X \chi'_i(\tuple w)$ if and only if $(M, X(\tuple w)) \models \chi_i(R)$.
But $\D(R)$ is equivalent to $\bigvee_{i=1}^t \chi_i(R)$. Therefore $(M, X(\tuple w)) \models \D(R)$ if and only if there exists some $i \in 1 \ldots t$ such that $\M \models_X \chi'_i(\tuple w)$, that is, if and only if $\M \models_X \bigsqcup_{i=1}^t \chi'_i(\tuple w)$; and thus $\D$ is definable in $\FO(=\!\!(\cdot), \NE, \sqcup)$. 
\end{proof}
\begin{corollary}
Let $\DD$ be a family of domain-independent dependencies each of which is union closed, upwards closed, or downwards closed\footnote{Or  non-jumping.}. Then $\DD$ is strongly first order if and only if all $\D \in \DD$ are individually strongly first order. 
\end{corollary}
\begin{proof}
The left to right direction is obvious. 

For the right to left direction, suppose that all $\D \in \DD$ are individually strongly first order. By Theorem \ref{thm:nonjumping}, every downwards closed (or non-jumping) $\D \in \DD$ is definable in $\FO(=\!\!(\cdot), \DD^\uparrow, \sqcup)$, and by Corollary \ref{coro:SFO_UC_def} every union closed $\D \in \DD$ is definable in $\FO(=\!\!(\cdot), \NE, \sqcup)$, and all upwards closed first order dependencies (including all $\NE$ dependencies) are in $\DD^\uparrow$. Therefore, every sentence of $\FO(\DD)$ is equivalent to some sentence of $\FO(=\!\!(\cdot), \DD^\uparrow, \sqcup)$; but by Proposition \ref{propo:up_const_global} every such sentence is equivalent to some first order sentence, and therefore $\DD$ is strongly first order. 
\end{proof}
\section{Conclusion}
In this work we found a full characterization of strongly first order dependencies for the case of dependencies that are also domain-independent and union-closed.  This essentially completes the characterization of strongly first order dependencies for all the main classes of dependencies studied in the literature. 

A general characterization, even modulo domain independence, of strongly first order dependencies is however still missing; and it is the hope of the author that the ideas and techniques developed in this work may be of some use for proving such a characterization. 
Additionally, it would be interesting to explore the same question with respect to richer languages than First Order Logic, as it was done in \cite{galliani2021doubly} for the case of First Order Logic (with Team Semantics) plus a restricted form of contradictory negation. For example, which dependencies can be added to Inclusion Logic without increasing its expressive power? This is not fully known at the moment, although it is known that no first order union-closed dependency increases the expressive power of Inclusion Logic \cite{gallhella13}. 

Finally, there exist quantitative variants of Team Semantics such as \emph{Probabilistic Team Semantics} \cite{durand2018probabilistic} and \emph{Multiteam Semantics} \cite{gradel2022logics}. The same question about which dependencies increase the expressive power of First Order Logic and which ones do not can be also asked for these interesting generalizations of the basic framework of Team Semantics, and its answer would also in their case provide a good starting point for the classification of the corresponding extensions of First Order Logic. 
\bibliographystyle{splncs04}
\bibliography{biblio}
\newpage
\appendix
\section*{Appendix}
In this appendix I report proof details that have not been included in the article. 
\subsection*{Preservation of Closure Properties}
In Section \ref{sect:prelims} it was stated that if all $\D \in \DD$ are downwards closed [union closed, with the empty team property] then all formulas of $\FO(\DD)$ have this property. 

This can be proven by structural induction on $\phi$. The base cases (i.e. if $\phi$ is a dependence atom $\D \tuple v$ or a first order literal $\alpha$) are obvious (note that by Proposition \ref{propo:flat}, first order literals are downwards closed \emph{and} union closed \emph{and} have the empty team property). 

The other cases are as follows: 
\begin{itemize}
    \item Let $\phi = \phi_1 \vee \phi_2$. 
    \begin{description}
        \item[Empty Team:] By induction hypothesis, for all models $\M$ we have that $\M \models_\emptyset \phi_1$ and $\M \models_\emptyset \phi_2$. Then since $\emptyset = \emptyset \cup \emptyset$, $\M \models_\emptyset \phi_1 \vee \phi_2$.
        \item[Downwards Closure:] Suppose that $\M \models_X \phi$ and $Y \subseteq X$. Then $X = X_1 \cup X_2$, where $\M \models_{X_1} \phi_1$ and $\M \models_{X_2} \phi_2$. Now let $Y_1 = Y \cap X_1$ and $Y_2 = Y \cap X_2$: by induction hypothesis, $\M \models_{Y_1} \phi_1$ and $\M \models_{Y_2} \phi_2$, and so --- since $Y = Y_1 \cup Y_2$ -- $\M \models_Y \phi_1 \vee \phi_2$.
        \item[Union Closure:] Suppose that $\M \models_{X_i} \phi$ for all $i \in I$. Then every $X_i$ is of the form $Y_1 \cup Z_i$, where $\M \models_{Y_i} \phi_1$ and $\M \models_{Z_i} \phi_2$. Now let $Y = \bigcup_{i \in I} Y_1$ and $Z = \bigcup_{i \in I} Z_i$: then by induction hypothesis $\M \models_Y \phi_1$ and $\M \models_Z \phi_2$. Therefore, for $X = \bigcup_{i \in I} X_i = Y \cup Z$, $\M \models_X \phi_1 \vee \phi_2$ as required.
    \end{description}
    \item Let $\phi = \phi_1 \wedge \phi_2$. 
    \begin{description}
        \item[Empty Team:] By induction hypothesis $\M \models_\emptyset \phi_1$ and $\M \models_\emptyset \phi_2$, and so $\M \models_\emptyset \phi_1 \wedge \phi_2$. 
        \item[Downwards Closure:] If $\M \models_X \phi_1$ and $\M \models_X \phi_2$ and $Y \subseteq X$, by induction hypothesis $\M \models_Y \phi_1$ and $\M \models_Y \phi_2$. But then $\M \models_Y \phi_1 \wedge \phi_2. $
        \item[Union Closure:] If $\M \models_{X_i} \phi$ for all $i \in I$, $\M \models_{X_i} \phi_1$ and $\M \models_{X_i} \phi_2$ for all such $i$. Now let $X = \bigcup_{i \in I} X_i$: by induction hypothesis, $\M \models_X \phi_1$ and $\M \models_X \phi_2$, and so $\M \models_X \phi_2 \wedge \phi_2$. 
    \end{description}
    \item Let $\phi = \exists v \psi$. 
    \begin{description}
        \item[Empty Team:] By induction hypothesis, $\M \models_\emptyset \psi$. Now let $H_0$ be the trival choice function $\emptyset \rightarrow \mathcal P(M) \backslash \{\emptyset\}$: then $\emptyset[H_0/v] = \emptyset$, and so $\M \models_{\emptyset[H_0/v]} \psi$, and so $\M \models_\emptyset \exists v \psi$.
        \item[Downwards Closure:] If $\M \models_X \exists v \psi$ then there exists some $H: X \rightarrow \mathcal P(M) \backslash \{\emptyset\}$ such that $\M \models_{X[H/v]} \psi$. Now if $Y \subseteq X$, let $H_{|Y}: Y \rightarrow \mathcal P(M) \backslash \{\emptyset\}$ be the restriction of $H$ to $Y$: then $Y[H_{|Y}/v] = \{s[m/v]: s \in Y, m \in H_{|Y}(s)\} \subseteq X[H/v]$, and so by induction hypothesis $\M \models_{Y[H_{|Y}/v} \psi$ and finally $\M \models_Y \exists v \psi$. 
        \item[Union Closure:] Suppose that $\M \models_{X_i} \exists v \psi$ for all $i \in I$. Then for all $i \in I$ there exists some $H_i: X_i \rightarrow \mathcal P(M) \backslash \{\emptyset\}$ such that $\M \models_{X_i[H_i/v]} \psi$. Now let $X = \bigcup_{i \in I} X_i$ and let $H: X \rightarrow \mathcal P(M) \backslash \{\emptyset\}$ be defined as 
        \[
            H(s) = \bigcup \{H_i(s) : i \in I, s \in X_i\}. 
        \]
        Then $X[H/v] = \{s[m/v]: \exists i \in I \text{ s.t. } s \in X_i \text{ and } m \in H_i(s)\} = \bigcup_{i \in I} \{s[m/v] : s \in X_i, m \in H_i(s)\} = \bigcup_{i \in I} X_i[H_i/v]$. 
        
        By induction hypothesis this implies that $\M \models_{X[H/v]} \psi$, and therefore $\M \models_X \exists v \psi$.
    \end{description}
    \item Let $\phi = \forall v \psi$. 
    \begin{description}
    \item[Empty Team:] By induction hypothesis, $\M \models_\emptyset \psi$. But $\emptyset[M/v] = \emptyset$, and so $\M \models_\emptyset \forall v \psi$. 
    \item[Downwards Closure:] If $\M \models_X \forall v \psi$ then $\M \models_{X[M/v]} \psi$. Now if $Y \subseteq X$, $Y[M/v] \subseteq X[M/v]$: so by induction hypothesis 
    \item[Union Closure:] If $\M \models_{X_i} \forall v \psi$ for all $i \in I$ then $\M \models_{X_i[M/v]} \psi$ for all such $i$. 
    Now let $X = \bigcup_{i \in I} X_i$: then $X[M/v] = \{s[m/v]: \exists i \in I \text{ s.t. } s \in X_i, m \in M\} = \bigcup_i X_i[M/v]$, and so by induction hypothesis $\M \models_{X[M/v]} \psi$ and finally $\M \models_X \forall v \psi$.
    \end{description}
\end{itemize}
\subsection*{Relativizability and Domain Independence}
Let $\D$ be any dependency and let $P$ be some fixed unary predicate symbol. Then, for all models $\M$ that contain the symbol $P$ in their signature, the semantics of the \emph{relativized dependency} $\D^{(P)}$ is defined by.the rule
\begin{description}
\item[TS-$\D^{(P)}$:] $\M \models_X \D^{(P)} \tuple x$ if and only if $(P^{\M}, X(\tuple x)) \in \D$. 
\end{description}
A dependency $\D$ is said to be \emph{relativizable} if every sentence of $\FO(\D^{(P)})$ (that is, of First Order Logic augmented with the relativized atoms $\D^{(P)} \tuple x$, for some \emph{fixed} choice of the predicate $P$) is equivalent to some sentence of $\FO(\D)$. 
\begin{proposition}
If $\D$ is domain independent, then it is relativizable. 
\end{proposition}
\begin{proof}
It suffices to show that $\D^{(P)} \tuple x$ is definable in $\FO(\D)$ as $\D \tuple x \wedge \bigwedge_{x \in \tuple x} P x$. 

Suppose that $\M \models_X \D^{(P)} \tuple x$. Then $(P^\M, X(\tuple x)) \in \D$. Then, since $\D$ is domain-independent, $(M, X(\tuple x)) \in \D$, and so $\M \models_X \D \tuple x$. Moreover, for $(P^\M, X(\tuple x))$ to be a model it must be the case that $X(\tuple x) \subseteq (P^\M)^{|\tuple x|}$, i.e., $s(x) \in P^\M$ for all $s \in X$ and all $x \in \tuple x$. Therefore $\M \models_X  \bigwedge_{x \in \tuple x} P x$ as well, and so $\M \models_X \D \tuple x \wedge \bigwedge_{x \in \tuple x} P x$. 

Conversely, suppose that $\M \models_X \D \tuple x \wedge \bigwedge_{x \in \tuple x} P x$.  Then $(M, X(\tuple x)) \in \D$ and $X(\tuple x) \subseteq (P^\M)^{|\tuple x|}$, and so since $\D$ is domain independent $(P^{\M}, X(\tuple x)) \in \D$ as well. 

Therefore, every sentence $\phi \in \FO(\D^{(P)})$ is equivalent to the sentence $\phi' \in \FO(\D)$ obtained by replacing each relativized atom $\D^{(P)} \tuple x$ with the corresponding $\D \tuple x \wedge \bigwedge_{x \in \tuple x} P x$. 
\end{proof}
\subsection*{Non-Jumping Dependences and Downwards Closure}

\begin{definition}[$\D_{\max}$]
Let $\D$ be a dependency. Then $(M, R) \in \D_{\max}$ if and only if $(M, R) \in \D$ and there is no $S \supsetneq R$ such that $(M, S) \in \D$.
\end{definition}

\begin{proposition}[\cite{galliani2019nonjumping}, Proposition 22]
If $\D$ is strongly first order and $(M, R) \in \D$ then there exists some $R' \supseteq R$ such that $(M, R') \in \D_{\max}$. 
\label{propo:max_exists}
\end{proposition}
\begin{definition}[Non-Jumping Dependency]
A dependency $\D$ is non-jumping if and only if, whenever $(M, R) \in \D$, there exists some $R' \supseteq R$ such that 
\begin{itemize}
    \item $(M, R') \in \D_{\max}$; 
    \item For all relations $S$ with $R \subseteq S \subseteq R'$, $(M, S) \in \D$. 
\end{itemize}
\end{definition}
\begin{proposition}
If $\D$ is downwards closed and strongly first order then $\D$ is non-jumping. 
\end{proposition}
\begin{proof}
Suppose that $(M, R) \in \D$. Then by Proposition \ref{propo:max_exists} there is some $R' \supseteq R$ such that $(M, R') \in \D_{\max}$. Then $(M, R') \in \D$, and since $\D$ is downwards closed this implies that $(M, S) \in \D$ for all $S \subseteq R'$ (and so, in particular, for all $S$ with $R \subseteq S \subseteq R'$). 
\end{proof}
\end{document}